\documentclass[reqno]{amsart}

\usepackage{amscd,amsmath,amssymb,amsfonts,bm}
\usepackage{stmaryrd}
\usepackage{bbold}
\usepackage{times}
\usepackage[all]{xy}
\usepackage{mathrsfs}
\usepackage{graphicx}

\usepackage{amsthm}

\theoremstyle{plain}
\newtheorem{thm}{Theorem}
\newtheorem{lem}[thm]{Lemma}
\newtheorem{cor}[thm]{Corollary}
\newtheorem{prop}[thm]{Proposition}

\theoremstyle{definition}

\newtheorem*{ack}{Acknowledgment}
\newtheorem*{notat}{Notations and conventions}

\theoremstyle{remark}
\newtheorem*{remark}{Remark}

\newcommand{\Hom}{\operatorname{Hom}}

\newcommand{\Aut}{\operatorname{Aut}}

\newcommand{\soc}{\operatorname{soc}}

\newcommand{\Spec}{\operatorname{Spec}}

\newcommand{\trdeg}{\operatorname{tr\,deg}}

\newcommand{\ch}{\operatorname{char}}

\newcommand{\onto}{\twoheadrightarrow}

\newcommand{\into}{\hookrightarrow}

\newcommand{\Mat}{\operatorname{M}}

\DeclareMathOperator{\Rat}{Rat} \DeclareMathOperator{\Prim}{Prim}
\DeclareMathOperator{\Max}{Max}
\DeclareMathOperator{\PIdeg}{PI\,deg}
\DeclareMathOperator{\GKdim}{GK\,dim}

\renewcommand{\k}{\mathbb{k}}
\renewcommand{\phi}{\varphi}

\newcommand{\GSpec}{G\text{-}\!\Spec}

\newcommand{\GRat}{G\text{-}\!\Rat}

\newcommand{\G}{\mathcal{G}}
\newcommand{\cO}{\mathcal{O}}

\newcommand{\cC}{\mathcal{C}}
\newcommand{\cR}{\mathcal{R}}

\newcommand{\fp}{\mathfrak{p}}

\newcommand{\fa}{\mathfrak{a}}

\newcommand{\e}{\varepsilon}

\newcommand{\ZZ}{\mathbb{Z}}

\newcommand{\cat}[1]{\operatorname{\mathsf{#1}}}
\newcommand{\Mod}{\cat{Mod}}
\newcommand{\fgmod}{\cat{mod}}

\newcommand{\Rep}{\cat{Rep}}
\newcommand{\rep}{\cat{rep}}

\newcommand{\irr}{\cat{irr}}

\newcommand{\cen}{\mathcal{Z}}

\newcommand{\bq}{\mathbf{q}}
\newcommand{\bfm}{\mathbf{m}}

\newcommand{\byG}{\!\!:\!\! G}

\newdir{ >}{{}*!/-5pt/\dir{>}}

\renewcommand{\labelenumi}{(\alph{enumi})}

\begin{document}

%
%

\title[Rational group actions on affine PI-algebras]%
{Rational group actions on affine PI-algebras}

\author{Martin Lorenz}

\address{Department of Mathematics, Temple University,
    Philadelphia, PA 19122}

\email{lorenz@temple.edu}

\urladdr{http://www.math.temple.edu/$\stackrel{\sim}{\phantom{.}}$lorenz}

\thanks{Research of the author supported in part by NSA Grant H98230-09-1-0026}

\subjclass[2010]{}


\maketitle

\begin{abstract}
Let $R$ be an affine PI-algebra over an algebraically closed field $\k$ and 
let $G$ be an affine algebraic $\k$-group that acts rationally by algebra automorphisms on $R$.
For $R$ prime and $G$ a torus, we show that $R$ has only finitely many 
$G$-prime ideals if and only if the action of 
$G$ on the center of $R$ is multiplicity free. This extends a standard result on affine algebraic
$G$-varieties. Under suitable hypotheses on $R$ and $G$, we also prove
a PI-version of a well-known result on spherical varieties
and a version of Schelter's catenarity theorem for $G$-primes.
\end{abstract}



\section{Introduction}

\subsection{} \label{SS:question}
This article 
addresses the 
following general question:
\begin{quote}
Suppose a group $G$ acts by automorphisms on a ring $R$.
When is the set $\GSpec R$ consisting
of all $G$-prime ideals of $R$ is finite?
\end{quote}
Recall that a proper $G$-stable (two-sided) ideal $I$ of $R$ is called
\emph{$G$-prime} if $AB \subseteq I$ for $G$-stable ideals $A$ and
$B$ of $R$ implies that $A \subseteq I$ or $B \subseteq I$.
For a noetherian algebra $R$ over a field $\k$ and
an algebraic $\k$-torus $G$ acting rationally on $R$ by $\k$-algebra automorphisms,
the above question was stated as Problem II.10.6  in \cite{kBkG02}.


\subsection{} \label{SS:motivation}
Now assume that $R$ is an associative algebra over an algebraically closed base field $\k$
and $G$ is an affine algebraic $\k$-group 
that acts rationally by $\k$-algebra automorphisms on $R$; 
see \ref{SS:ratact} below for a brief reminder on rational actions. 

The question in \ref{SS:question} is motivated in part by the
stratification of  $\Spec R$ that is induced by the action of $G$.
Namely, there is a surjection
\begin{equation*} \label{E:gamma}
\Spec R \onto \GSpec R 
\end{equation*}
sending $P$ to the largest
$G$-stable ideal of $R$ that is contained in $P$, 
and \cite[Theorem  9]{mL09} gives
a precise description of the fibers of this map in terms
of \emph{commutative} algebras.
Hence, from a noncommutative perspective, 
the focus shifts to the description of $\GSpec R$, with finiteness being the optimal scenario.
It turns out that, as long as the deformation parameters are
chosen in a sufficiently generic manner, $\GSpec R$ 
is indeed finite for all quantized coordinate
algebras $R = \cO_\bq(X)$ that have been analyzed in detail thus far, 
the acting group $G$ typically being a suitably chosen
algebraic torus.
Notable examples include the (generic) quantized coordinate rings of
all semisimple algebraic groups (Joseph \cite{aJ95}, Hodges,
Levasseur and Toro \cite{tHtLmT97}), quantum matrices and quantum
Grassmannians (Cauchon, Lenagan and others; e.g, \cite{gC03a},
\cite{gC03b} and \cite{LLR08}). 
Finiteness of $\GSpec R$ has also been observed in Leavitt
path algebras $R$, again for the action 
of a suitable torus $G$ \cite{ABR}.
These finiteness results all
depend either on finding a presentation of
$R$ as an iterated skew polynomial algebra, a class of algebras for which finiteness has
been established by
Goodearl and Letzter \cite{kGeL94}, \cite{kGeL00}, or else on long
calculations in $R$. A general finiteness criterion for $\GSpec R$
is currently lacking. 


\subsection{} \label{SS:motivation2}
Our main focus in this note will be on the case where $R$ is an affine PI-algebra over $\k$.
This will be assumed for the remainder of the Introduction, and
$G$ will be an affine algebraic $\k$-group that acts rationally by $\k$-algebra automorphisms 
on $R$ as in \ref{SS:motivation}. 

In order to give the finiteness problem in \ref{SS:question} 
a geometric perspective, we mention
the following connection with $G$-orbits of rational ideals. 
Here, a prime ideal $P$ of $R$ is called \emph{rational} if 
$\cC(R/P) = \k$, where $\cC(\,.\,)$ denotes the center of the classical ring of quotients. Rational primes
are exactly the closed points of $\Spec R$;
see \ref{SSS:DM} below for several equivalent characterizations of rationality.
An ideal $P \in \GSpec R$ is said to be \emph{$G$-rational} if
$\cC(R/P)^G = \k$. 
The subset of $\Spec R$ consisting of all rational primes of $R$ will be denoted by 
$\Rat R$, and $\GRat R$ will denote the subset of $\GSpec R$ consisting of all $G$-rational ideals.
Since $R$ satisfies the ascending chain condition for 
semiprime ideals (\ref{SSS:semiprime}), the Nullstellensatz (\ref{SSS:Nst}) and the 
Dixmier-M{\oe}glin equivalence (\ref{SSS:DM}), the following proposition is 
a special case of \cite[Proposition 14]{mL09}.

\begin{prop} \label{P:finite}
Let $R$ be an affine PI-algebra over the algebraically closed field $\k$ and 
let $G$ be an affine algebraic $\k$-group that acts rationally by $\k$-algebra automorphisms on $R$.
Then the following are equivalent:
\begin{enumerate}
\renewcommand{\labelenumi}{(\roman{enumi})}
\item $\GSpec R$ is finite;
\item $\GRat R$ is finite;
\item $G$ has finitely many orbits in $\Rat R$;
\item $\GRat R = \GSpec R$.
\end{enumerate}
\end{prop}

\noindent
Thus, the problem at hand amounts to determining when all $G$-primes of $R$
are $G$-rational.

\subsection{} \label{SS:content}
In studying the finiteness question \ref{SS:question} we may assume without loss that $G$ is connected. 
In this case,
all $G$-primes of $R$ are actually prime, and hence $\GSpec R$ is  the set of all $G$-stable prime
ideals of $R$; see Lemma~\ref{L:ratact} below.
The main result of this note concerns the special case where $G$ is a torus; it extends a standard result 
on affine algebraic $G$-varieties \cite[II.3.3 Satz 5]{hK84} to PI-algebras.

\begin{thm} \label{T:torus}
Let $R$ be a prime affine PI-algebra over the algebraically closed field $\k$ and 
let $G$ be an algebraic $\k$-torus that acts rationally by $\k$-algebra automorphisms on $R$.
Then $\GSpec R$ is finite if and only if the action of $G$ on the center
$\cen(R)$ is multiplicity free.
\end{thm}

\noindent
Here, multiplicity freeness means that, for each rational
character $\lambda \colon G \to \k^\times$, the weight space $\cen(R)_\lambda = 
\{ r \in \cen(R) \mid g.r = \lambda(g)r \text{ for all $g \in G$} \}$ has dimension at most $1$.

The proof of Theorem~\ref{T:torus} will be given in Section~\ref{S:main} after deploying some auxiliary results
and a generous amount of background material in Section~\ref{S:prelim}. We remark that, when $R$ is
also assumed noetherian, Theorem~\ref{T:torus} is quite a bit easier, being an immediate consequence 
of Proposition~\ref{P:TR} and Lemma~\ref{L:PIbounded}(b) below. We conclude, in Section~\ref{S:noethPI},
with two results for noetherian $R$, namely a PI-version of a standard result on spherical varieties
(Proposition~\ref{P:rationalPI}) and a version of Schelter's catenarity theorem for $G$-primes
(Proposition~\ref{P:catenary}).

\begin{notat} 
All rings have a $1$ which is inherited by subrings and preserved under
homorphisms. 
The action of the group $G$ on the ring $R$ will be written as
$G \times R \to R$, $(g,r) \mapsto g.r$.
For any ideal $I$ of $R$, we will write $I\byG = \bigcap_{g \in G} g.I$; this is the largest
$G$-stable ideal of $R$ that is contained in $I$. The symbol $\subset$ denotes a proper inclusion.
\end{notat}


\section{Preliminaries} \label{S:prelim}


\subsection{Finite centralizing ring extensions} \label{SS:centralizing}

A ring extension  $R\subseteq S$  is called \emph{centralizing} 
if $S=R\mathbf{C}_S(R)$ where $\mathbf{C}_S(R) = \{s \in S \mid
sr = rs \text{ for all $r \in R$}\}$. In this case, for any prime ideal $P$ of $S$,
the contraction $P \cap R$ is easily seen to be a prime ideal of $R$.
A centralizing extension $R\subseteq S$  is called \emph{finite}, if $S$ is finitely generated as
left or, equivalently, right $R$-module. By results of G. Bergman \cite{gBx}, \cite{gBxx} (see also \cite{RS81}), 
the classical relations of lying over and incomparability for prime ideals hold in any finite centralizing 
extension $R\subseteq S$:
\begin{itemize}
\item given $Q \in \Spec R$, there exists $P \in \Spec S$ such that $Q = P \cap R$ \quad (Lying Over);
\item if $P,P' \in \Spec S$ are such that $P \subset P'$
then $P \cap R \subset P' \cap R$ \quad (Incomparability).
\end{itemize}

\begin{lem} \label{L:centralizing}
Let $R\subseteq S$ be a finite centralizing extension of 
rings and let $G$ be a group acting
by automorphisms on $S$ that stabilize $R$. Assume that every ideal $A$ of $S$ contains a finite
product of primes each of which contains $A$.
Then contraction
yields a surjective map
\begin{equation*} 
\GSpec S \onto \GSpec R \,,  \qquad I \mapsto I \cap R
\end{equation*}
with finite fibers. In particular, if one of $\GSpec S$ or $\GSpec R$ is finite then so is the other.
\end{lem}

\begin{proof}
First, we note that the $G$-primes of $S$ are exactly the ideals 
of the form $P \byG$ with $P \in \Spec S$. Indeed, it is straightforward to check that
$P \byG$ is $G$-prime. Conversely, for any given $I \in \GSpec S$, there are finitely
many $P_i \in \Spec S$ (not necessarily distinct) with $I \subseteq P_i$ 
and $\prod_i P_i \subseteq I$.
But then $I \subseteq P_i\byG$ for each $i$ 
and $\prod_i P_i\byG \subseteq I$, whence $I = P_i\byG$ for some $i$.
In particular,
each $I \in \GSpec S$ is semiprime. The group $G$ permutes the finitely many
primes of $S$ that are minimal over $I$ and $G$-primeness forces these primes to form a single
$G$-orbit. Therefore, we may write $I = P\byG$ with $P \in \Spec S$ having a finite $G$-orbit.
Similar remarks apply to the ring $R$, because every ideal $B$ of $R$ also contains a finite
product of primes each of which contains $B$; this follows from the fact that $B$ contains
some finite power of $BS\cap R$ by \cite[Corollary 1.4]{mL81}.

Now let $I \in \GSpec S$ be given and let $A,B$ be $G$-stable ideals of $R$ such that 
$AB \subseteq I \cap R$. Then $AS = SA$ is a $G$-stable ideal of $S$ and similarly for $B$.
Since $(AS)(BS) = ABS \subseteq I$, we must have $AS \subseteq I$ or $BS \subseteq I$ and
hence $A \subseteq I \cap R$ or $B \subseteq I \cap R$. Thus contraction yields a 
well-defined map $\GSpec S \to \GSpec R$\,.

For surjectivity of the contraction map, let $J \in \GSpec R$ be given and write $J = Q\byG$ with
$Q \in \Spec R$. By Lying Over we may choose $P \in \Spec S$ with $Q = P\cap R$. Putting
$I = P\byG$ we obtain a $G$-prime of $S$ such that $J = I \cap R$\,.

Finally, assume $I \in \GSpec S$ contracts to a given $J \in \GSpec R$. 
Write $I = P \byG$ with $P \in \Spec S$ having a finite $G$-orbit. 
We claim that $P$ must be minimal over the ideal $JS$. Indeed, suppose that $JS \subseteq P'
\subset P$ for some $P' \in \Spec S$. Then Incomparability gives
$P \cap R \supset P' \cap R \supseteq J = \bigcap_{g\in G} g.(P \cap R)$.
Since this intersection is finite and $P'\cap R$ is prime, we conclude that
$g.(P \cap R)  \subseteq P' \cap R$ for some $g\in G$. Hence, $g.(P \cap R) \subset P \cap R$
which is impossible. This proves minimality of $P$ over $JS$. 
It follows that there are finitely many possibilities for $P$,
and hence there are finitely many possibilities for $I$. This completes the proof of the lemma.
\end{proof}


The hypothesis that every ideal of $S$ contains a finite
product of prime divisors is of course satisfied, by Noether's classical argument, 
if $S$ satisfies the ascending chain condition for ideals. More importantly for our purposes, 
the hypothesis also holds
for any affine PI-algebra $S$ over some commutative noetherian ring by Braun's Theorem \cite[6.3.39]{lR88}.


\subsection{Rational group actions} \label{SS:ratact}

Let $G$ be an affine algebraic $\k$-group, where $\k$ is an algebraically closed field, and let $\k[G]$ denote the
Hopf algebra of regular functions on $G$. A $\k$-vector space $M$ is called a 
\emph{$G$-module} if 
$M$ is a $\k[G]$-comodule; see Jantzen
\cite[2.7-2.8]{jcJ03} or Waterhouse \cite[3.1-3.2]{wW79}. Writing the comodule structure map
$\Delta_M \colon M \to M \otimes \k[G]$ as 
$\Delta_M(m) = \sum m_0 \otimes m_1$, the group
$G$ acts by $\k$-linear transformations on $M$ via
\begin{equation*} \label{E:action}
g.m  = \sum m_0 m_1(g) \qquad (g \in G, m \in M)\ .
\end{equation*}
Such $G$-actions, called \emph{rational $G$-actions}, are in particular locally finite: the $G$-orbit
of any $m\in M$
is contained in the finite-dimensional $\k$-subspace of $M$ that is generated by $\{ m_0\}$. 
If $G$ acts rationally on $M$ then it does so on all $G$-subquotients of $M$. 
Moreover, every irreducible $G$-submodule of $M$ is finite-dimensional, and the sum
of all irreducible $G$-submodules is an essential $G$-submodule of $M$, called the \emph{socle}
of $M$ and denoted by $\soc_G M$. In the following, we will denote the set
of isomorphism classes of irreducible $G$-modules by $\irr G$ and,
for each $E \in \irr G$, we let 
\[
[R:E] \in \ZZ_{\ge 0} \cup \{\infty\} 
\]
denote the \emph{multiplicity} of $E$ in $R$; see \cite[I.2.14]{jcJ03}.

We will be primarily concerned with the situation where $G$ acts rationally by algebra automorphisms 
on a $\k$-algebra $R$. This is equivalent to $R$ being a right $\k[G]$-comodule algebra in the
sense of \cite[4.1.2]{sM93}.
In the special case where $G \cong (\k^\times)^d$ is an
algebraic torus, we have $\irr G = X(G) \cong \ZZ^d$, the lattice of rational
characters $\lambda \colon G \to \k^\times$. We will usually write $\Lambda = X(G)$.
A rational $G$-action on $R$
is equivalent to a $\ZZ^d$-grading $R = \bigoplus_{\lambda \in \ZZ^d} R_{\lambda}$ of the algebra $R$.
This follows from the fact that $\k[G]$ is the group algebra $\k \Lambda$
of the lattice $\Lambda  \cong \ZZ^d$, and $\k\Lambda$-comodule algebras
are the same as $\Lambda$-graded algebras; see \cite[4.1.7]{sM93}. The homogeneous
component of $R$ of degree $\lambda$ is the weight
space
\[
R_\lambda = \{ r \in R \mid g.r = \lambda(g)r \text{ for all $g \in G$} \}\ ,
\]
and $[R:\lambda] = \dim_\k R_\lambda$.

The following lemma was referred to in the Introduction.

\begin{lem} \label{L:ratact}
Let $R$ be a $\k$-algebra, where $\k$ is an algebraically closed field, and let
$G$ be an affine algebraic $\k$-group that acts rationally by $\k$-algebra automorphisms on $R$.
If $G^0 \subseteq G$ 
denotes the connected component 
of the identity, then $G^0\text{-}\!\Spec R$ consists of the ordinary prime ideals of $R$ 
that are $G^0$-stable.
Moreover, $\GSpec R$ is finite if and only if $G^0\text{-}\!\Spec$ is 
finite. 
\end{lem}

\begin{proof}
For the assertion that all $G^0$-primes are prime, see \cite[Proposition 19(a)]{mL08}.

The second assertion, that $\GSpec R$ is finite if and only if $G^0\text{-}\!\Spec$ is so,
actually holds for any (normal) subgroup $N \trianglelefteq G$ having finite index in $G$ in
place of $G^0$. Putting $\G = G/N$, we first note that the $G$-primes of $R$ are exactly the ideals of the 
form $P = \bigcap_{x\in\G} x.Q$ with $Q \in N\text{-}\!\Spec R$. Indeed, $\bigcap_{x\in\G} x.Q$ is easily seen to
be $G$-prime. Conversely, 
any $P \in \GSpec R$ has the form $P = P'\byG$ with $P' \in \Spec R$ by \cite[Proposition 8]{mL08},
and hence we may take $Q = P'\byG^0$. Moreover, the intersection $\bigcap_{x\in\G} x.Q$ determines
the $N$-prime $Q$ to within $\G$-conjugacy, because all $x.Q$ are $N$-prime ideals of $R$ and $\G$ is finite.
Therefore, finiteness of $N\text{-}\!\Spec R$ is equivalent to finiteness of $\GSpec R$. 
\end{proof}


\subsection{Some ring theoretic background on affine PI-algebras} \label{SS:background}

Let $R$ be an affine PI-algebra over a commutative noetherian ring $\k$. The following
facts are well-known. 

\subsubsection{Semiprime ideals} \label{SSS:semiprime}
The ring $R$ satisfies the ascending chain condition for 
semiprime ideals and, for each ideal
$I$ of $R$, there are only finitely many primes of $R$ that are minimal over $I$. 
If $I$ is semiprime then $R/I$ is a right and left
Goldie ring and the extended centroid of $R/I$, in the sense of Martindale \cite{wM69}, is given by 
$\cC(R/I) = \cen(Q(R/I))$,
the center of the classical ring of quotients of $R/I$. If $I$ is prime then $\cC(R/I)$ is
identical to the field of fractions of  $\cen(R/I)$ by Posner's Theorem.
See \cite[6.1.30, 6.3.36']{lR88}, \cite[13.6.9]{jMcCjR87}, \cite[1.4.2]{mL08} for all this.

\subsubsection{$G$-prime ideals} \label{SSS:Gprime}
By Braun's Theorem \cite[6.3.39]{lR88}, every ideal $I$ 
of $R$ contains a finite product of primes that contain $I$. 
As in the proof of Lemma~\ref{L:centralizing} it follows that, for any group 
$G$ acting by ring automorphisms on $R$, the 
$G$-primes of $R$ are exactly the ideals of the form $P\byG$ with $P\in \Spec R$, and
$P$ can be chosen to have a finite $G$-orbit.
In particular, every $I \in \GSpec R$ is semiprime. The ring of $G$-invariants
$\cC(R/I)^G$ is a field for every $I \in \GSpec R$; see \cite[Prop. 9]{mL08}. 

\subsubsection{Nullstellensatz}  \label{SSS:Nst}
If $\k$ is a Jacobson ring then so is $R$: every prime ideal of $R$ 
is an intersection of primitive ideals. Moreover, if $P$ is a primitive ideal of $R$ then $P$
is maximal; in fact, $\k/P\cap\k$ is a field and $R/P$ is a finite-dimensional algebra over this field. See
\cite[6.3.3]{lR88}.

\subsubsection{Rational ideals and the Dixmier-M{\oe}glin equivalence}  \label{SSS:DM}
Now assume that $\k$ is an algebraically closed field.
Recall that a prime ideal $P$ of $R$ is said to be \emph{rational} if 
$\cC(R/P) = \k$ or, equivalently, $\cen(R/P) = \k$. By Posner's Theorem \cite[6.1.30]{lR88}, this forces $P$ to 
have finite $\k$-codimension in $R$.
In fact, for any prime ideal of $R$, the following properties coincide (\emph{Dixmier-M{\oe}glin equivalence}),
the implications $\Rightarrow$ being either trivial or immediate from the Nullstellensatz:
\[
\text{finite codimensional} \equiv \text{maximal} \equiv \text{locally closed in $\Spec R$} 
\equiv \text{primitive} \equiv \text{rational}\ .
\]


\subsection{The trace ring of a prime PI-ring} \label{SS:TR}

Let $R$ be a prime PI-ring with center $C =  \cen(R)$. By Posner's Theorem  \cite[6.1.30]{lR88}, 
the central localization
$Q(R) = R_{C \setminus \{ 0 \}}$ is a central simple algebra over the field of fractions $F = Q(C) = \cC(R)$. 
For each $q \in Q(R)$ we can consider the 
\emph{reduced characteristic polynomial}
$c_q(X) \in F[X]$. In detail, letting $F^{\text{alg}}$ denote an algebraic closure of $F$, we have
an isomorphism of $F^{\text{alg}}$-algebras 
\begin{equation} \label{E:TR}
\phi \colon Q(R)\otimes_F F^{\text{alg}} \cong \Mat_n(F^{\text{alg}})
\end{equation}
for some $n$. 
This isomorphism allows us to define $c_q(X)$ as the characteristic
polynomial of the matrix $\phi(q\otimes 1) \in \Mat_n(F^{\text{alg}})$. 
One can show that $c_q(X)$ has coefficients in $F$ and is independent of 
the choice of the isomorphism $\phi$; 
see \cite[\S 9a]{iR75}  or \cite[\S12.3]{nB73}.

The \emph{commutative trace ring} of $R$, by definition, is the $C$-subalgebra of $F$ that is generated by
the coefficients of all polynomials $c_r(X)$ with $r \in R$; this algebra will be denoted by $T$.
The \emph{trace ring} of $R$, denoted by $TR$, is the $C$-subalgebra of $Q(R)$ that is 
generated by $R$ and $T$\,. 
The following result is standard; see \cite[13.9.11]{jMcCjR87} or \cite[3.2]{nVE89}.

\begin{lem} \label{L:TR}
Let $R$ be a prime PI-ring that is an affine algebra over
some commutative noetherian ring $\k$. Then $T$ is an affine commutative $\k$-algebra
and $TR$ is a finitely generated $T$-module. Furthermore, $TR$ is finitely generated as $R$-module
if and only if $R$ is noetherian.
\end{lem}

Now suppose that a group $G$
acts by ring automorphisms on $R$. The action of $G$ extends uniquely to an action on the
trace ring $TR$, and this action stabilizes $T$. 
To see this, note that the $G$-action on $R$ extends uniquely to an action on the
ring of fractions $Q(R)$.
Each $g \in G$ stabilizes $F = \cen(Q(R))$, and
hence $g$ yields an automorphism of $F[X]$ via its action on
the coefficients of polynomials. The reduced characteristic polynomials of $q \in Q(R)$
and of $g.q$ are related by
\begin{equation} \label{E:groupactions}
c_{g.q}(X) = g.c_q(X)\ .
\end{equation}
Indeed, extending $g$ to a field automorphism of $F^{\text{alg}}$, we obtain automorphisms
$\Mat_n(g) \in \Aut \Mat_n(F^{\text{alg}})$ and $\alpha_g \in \Aut Q(R)\otimes_F F^{\text{alg}}$,
the latter being defined by $\alpha_g(q \otimes f) = g.q \otimes g.f$. Fixing $\phi$ as in 
\eqref{E:TR} we obtain an isomorphism of $F^{\text{alg}}$-algebras 
$\Mat_n(g)^{-1} \circ \phi \circ \alpha_g \colon Q(R)\otimes_F F^{\text{alg}} 
\cong \Mat_n(F^{\text{alg}})$.
Using this isomorphism to compute reduced characteristic polynomials, we see that
$c_q(X) = g^{-1}.c_{g.q}(X)$, proving \eqref{E:groupactions}. Since $g.r \in R$ for $r \in R$,
equation \eqref{E:groupactions} shows that the commutative trace ring $T$ is stable under the
action of $G$ on $Q(R)$, and hence so is the trace ring $TR$. For rational actions, we have 
the following result of Vonessen \cite[Proposition 3.4]{nVE89}.

\begin{lem}[Vonessen] \label{L:TR2}
Let $R$ a prime PI-algebra over an algebraically closed
field $\k$ and let $G$ be an affine algebraic $\k$-group that acts 
rationally by $\k$-algebra automorphisms on $R$. Then
the induced $G$-actions on $TR$ and on $T$ 
are rational as well. 
\end{lem}

In general, the finiteness problem \ref{SS:question} transfers nicely to trace rings.

\begin{prop} \label{P:TR}
Let $R$ be a prime PI-ring that is  an affine algebra over
some commutative noetherian ring. Let $G$ be a group acting by ring automorphism on $R$  
and consider the induced $G$-actions on $T$ and on $TR$. Then
$\GSpec T$ is finite if and only if $\GSpec TR$ is finite. If $R$ is noetherian, then this is
also equivalent to $\GSpec R$ being finite.
\end{prop}

\begin{proof}
Lemma~\ref{L:centralizing}, applied to the finite centralizing extension $T \subseteq TR$ (Lemma~\ref{L:TR}), 
tells us that finiteness of $\GSpec TR$ is equivalent to finiteness of $\GSpec T$. 
If $R$ is noetherian then we may argue in the same way for the finite centralizing extension $R \subseteq TR$. \end{proof}


\section{Main result} \label{S:main}

\emph{Throughout this section, $R$ denotes an affine PI-algebra over an algebraically closed field $\k$
and  $G$ will be an affine algebraic $\k$-group that acts rationally by $\k$-algebra automorphisms on $R$.}


\subsection{Suffient criteria for $G$-rationality} \label{SS:GRat}

By Proposition~\ref{P:finite} we know that $\GSpec R$ is finite if and only if all $G$-primes of $R$ are $G$-rational.
Therefore, $G$-rationality criteria are essential.
As usual, the algebra $R$ will be called $G$-prime if the zero ideal of $R$ is $G$-prime; similarly for
$G$-rationality. 

\begin{lem} \label{L:PIbounded}
Assume that $R$ is $G$-prime.
\begin{enumerate}
\item 
If there is an $N \in \ZZ$ such that $[\soc_G \cen(R):E] \le N$ for all $E \in \irr G$ then 
$R$ is $G$-rational.
\item
If $G$ is connected solvable then $R$ is $G$-rational if and only if 
$[\soc_G \cen(R):E] \le 1$ for all $E \in \irr G$.
\end{enumerate}
\end{lem}

\begin{proof}
(a) For a given $q \in \cC(R)^G$ put $I = \{ r \in R
\mid qr \in R \}$; this is a nonzero $G$-stable ideal of $R$. Therefore, $J = I^N \cap \cen(R)$ is a 
nonzero $G$-stable ideal of $\cen(R)$; see \cite[6.1.28]{lR88}. 
Note that $q^iJ \subseteq \cen(R)$ for $0 \le i \le N$.
We have $E \into J$ for
some $E \in \irr G$ and multiplication with $q^i$ yields a $G$-equivariant map
$E \into J \to \cen(R)$. 
Since $\dim_\k \Hom_{G}(E, \cen(R)) = [\soc_G \cen(R):E] \le N$, there are $k_i \in \k$, not all $0$,
such that $c = \sum_{i=0}^{N} k_iq^i$ annihilates $E$. 
But nonzero elements of $\cC(R)^G$ are regular in $Q(R)$; so we
must have $c=0$. Thus, $q$ is algebraic over $\k$ and so $q \in \k$.

(b)
The condition is sufficient by part (a).
For the converse, assume that $E_1 \oplus E_2 \subseteq \cen(R)$ for isomorphic $E_i \in \irr G$. 
By the Lie-Kolchin Theorem \cite[III.10.5]{aB91}, 
$E_i = \k x_i$ for suitable $x_i$. Since $x_i$ generates a $G$-stable two-sided ideal, $x_i$ is regular in $R$.
The quotient $x_1x_2^{-1} \in \cC(R)$ is a non-scalar
$G$-invariant; so $R$ is not $G$-rational.
\end{proof}

\begin{remark}
A simplified version of the argument in the proof of (a), without recourse to \cite[6.1.28]{lR88}, establishes the
following general fact:
Let $A$ be an arbitrary (associative) $\k$-algebra and let $G$ be a group that acts on $A$ by locally finite 
$\k$-algebra automorphisms.
If there is an $N \in \ZZ$ such that $[A:E] \le N$ for all finite-dimensional 
irreducible $\k G$-modules $E$ then 
$\GSpec A = \GRat A$.
\end{remark}


\subsection{Regular primes} \label{SS:regular}

Recall from \eqref{E:TR} that if $R$ is prime, then  the classical ring of quotients
$Q(R)$ is a central simple algebra 
over the field of fractions $F= Q(\cen(R))$. The \emph{PI degree} of $R$, by definition, is the degree of this
central simple algebra: 
$\PIdeg R = \sqrt{\dim_{F} Q(R)}$.
For any $P \in \Spec R$, one has $\PIdeg R/P \le \PIdeg R$. The prime $P$ is called \emph{regular}
if equality holds here. The regular primes form an open subset of $\Spec R$. See 
\cite[p.~104]{lR88} or \cite[13.7.2]{jMcCjR87} for all this.

Now let $G$ be an algebraic $\k$-torus.
In particular, $G$ is connected and so $\GSpec R$ consists of the $G$-stable prime ideals of $R$
by Lemma~\ref{L:ratact}.

\begin{lem} \label{L:regular}
Let $G$ be an algebraic $\k$-torus and assume that $R$ is prime. 
Then, for every regular $P \in \GSpec R$, we have $\trdeg_\k \cC(R/P)^G \le \trdeg_\k \cC(R)^G$.
Consequently, if $R$ is $G$-rational then all regular primes in $\GSpec R$ belong to $\GRat R$.
\end{lem}

\begin{proof}
Let $P \in \GSpec R$ be regular. 
Put $n = \PIdeg R$ and let $g_n(R)^+$ denote the Formanek center of $R$;
this is a $G$-stable ideal of 
$\cen(R)$ such that $g_n(R)^+ \nsubseteq P$; see \cite[6.1.37]{lR88} or \cite[13.7.2(i)]{jMcCjR87}. 
Therefore, we may
choose a semi-invariant $c \in g_n(R)^+_\lambda$ with $c \notin P$. The group $G$ acts rationally
on localization $R_c = R[1/c]$ and
$R_c$ is Azumaya by the Artin-Procesi Theorem \cite[13.5.14]{jMcCjR87}.
Therefore, $\cen(R_c)$ maps onto $\cen(R_c/PR_c)$ and  $\cen(R_c)_\lambda$ maps onto 
$\cen(R_c/PR_c)_\lambda$ for all $\lambda \in X(G)$. 
The map $\cen(R_c) \onto \cen(R_c/PR_c)$ extends to a
$G$-equivariant epimorphism $\cen(R_\fp) \onto \cC(R/P) =  Q(\cen(R_c/PR_c))$,
where $\fp = P \cap \cen(R)$. Since $\cen(R_\fp)^G \subseteq \cC(R)^G$, it suffices to show that
$\cen(R_\fp)^G$ maps onto $\cC(R/P)^G$.
But, given $q \in \cC(R/P)^G$, we can find a semi-invariant $0 \ne x \in \cen(R_c/PR_c)_\lambda$
such that $qx\in \cen(R_c/PR_c)$, and we can further find $y,z \in \cen(R_c)_\lambda$ with $y \mapsto x$
and $z \mapsto qx$. Then $zy^{-1} \in \cen(R_\fp)^G$ maps to $q$. This proves the lemma.
\end{proof}


\subsection{Proof of Theorem~\ref{T:torus}} \label{SS:proof}

Let $G$ be an algebraic $\k$-torus and assume that $R$ is prime.
We need to show that $\GSpec R$ is finite if and only if the action of $G$ on 
$\cen(R)$ is multiplicity free.
By Lemma~\ref{L:PIbounded}(b), 
the latter property is
equivalent to $G$-rationality of $R$, and this is certainly necessary for $\GSpec R$ to be finite
by Proposition~\ref{P:finite}.

Now assume that $R$ is $G$-rational. By Proposition~\ref{P:finite} we must show that 
all $G$-primes of $R$ are $G$-rational. Lemma~\ref{L:regular} ensures this for the regular
$G$-primes. In particular, we may assume that $n:= \PIdeg R >1$.
Now consider $P \in \GSpec R$ with $\PIdeg R/P < n$. Then $P$ contains
the ideal $\fa = g_n(R)R \subseteq R$; this is a nonzero $G$-stable common ideal of $R$ and of 
the trace ring $R':= TR$ of $R$; see \cite[6.1.37 and 6.3.28]{lR88}. All primes of $R$ that are
minimal over $\fa$ are $G$-stable. Let $Q$ be one of these primes such that $Q \subseteq P$.
It suffices to show that $Q$ is $G$-rational. For, then we may replace $R$ by $R/Q$, and since
$\PIdeg R/Q < n$, we may argue by induction that $P$ is $G$-rational. 

First, we claim that there exists $Q' \in \GSpec R'$ with $Q' \cap R = Q$. Indeed, choosing $Q'$ to be a
$G$-stable ideal of $R'$ that is maximal subject to the requirement that $Q' \cap R \subseteq Q$,
it is straightforward to see that $Q'$ is $G$-prime. If $Q' \cap R \neq Q$ then $Q' \nsupseteq \fa$ 
by minimality of $Q$ over $\fa$. Thus, $Q' + \fa$ is a $G$-stable ideal of $R'$ which properly
contains $Q'$ and yet also satisfies $(Q' + \fa) \cap R = (Q' \cap R) + \fa \subseteq Q$. 
Since this contradicts our maximal choice of $Q'$, we must have $Q' \cap R = Q$ as claimed.

Next, we show that $Q'$ is $G$-rational. To see this, 
recall from Lemma~\ref{L:TR2} that $G$ acts rationally on the trace rings $T$ and $TR$.
Moreover, $T$ is an affine commutative $\k$-algebra
that is $G$-rational, because $Q(T)^G = \cC(R)^G = \k$.
Therefore, by the case $n=1$, we know that $\GSpec T$ is finite.
By Proposition~\ref{P:TR}, $\GSpec TR$ is finite as well, and in view of Proposition~\ref{P:finite}, 
this forces $Q'$ to be $G$-rational.

Finally, we show that $Q$ is $G$-rational; this will finish the proof. But $\cC(R/Q) \subseteq
\cC(R'/Q')$ and $\cC(R'/Q')^G = \k$ by the foregoing. Therefore, $\cC(R/Q)^G = \k$ as
desired.



\section{Related results} \label{S:noethPI}

\emph{In this section, $R$ and $G$ are as in the previous section and $R$ is also
assumed noetherian.}


\subsection{Actions of reductive groups} \label{SS:rationalPI}

Recall from Lemma~\ref{L:TR2} that
the induced $G$-action on the commutative trace ring $T$ is rational. 
This enables us to quote results from algebraic geometry.

\begin{prop} \label{P:rationalPI}
Assume that $R$ is prime and that $G$ is connected reductive. 
Let $F = Q(\cen(R))$ denote the field of fractions of the center of $R$,
and let $F^B \subseteq F$ denote the invariant
subfield of a Borel subgroup $B \le G$. 
If $F^B = \k$ then $B\text{-}\!\Spec R$ is finite (and hence $\GSpec R$ is finite as well). 
\end{prop}

\begin{proof}
By Proposition~\ref{P:TR},  $B\text{-}\!\Spec R$ is finite if and only if $B\text{-}\!\Spec T$ is finite. Now, $T$
is an affine commutative domain over $\k$ and the field of fractions of $T$ is $F$.
By a standard result on spherical varieties \cite[Corollary 2.6]{fK95}, 
the condition $F^B = \k$ implies that 
there are only finitely many $B$-orbits in $\Rat T$.
The latter fact is equivalent to finiteness of $B\text{-}\!\Spec T$ by Proposition~\ref{P:finite}. 
This proves the proposition.
\end{proof}


\subsection{Catenarity} \label{SS:catenarity}

A partially ordered set $(P,\le)$ is said to be \emph{catenary} if, given any two $x < x'$ in $P$,  all
saturated chains $x = x_0 < x_1 < \dots < x_r = x'$ have the same finite length $r = r(x,x')$.
The following observation, for commutative algebras, goes back to conversations that
I had with R. Rentschler a long time ago; cf. \cite[\S 3]{mpM79}.
As usual, we let $\GKdim$ denote Gelfand-Kirillov dimension.

\begin{prop} \label{P:catenary}
If the connected component of the identity of $G$ solvable then 
the poset $(\GSpec R,\subseteq)$ is catenary. In fact, every
saturated chain $Q=Q_0 \subset Q_1 \subset \dots
\subset Q_r=Q'$ in $\GSpec R$ has length $r = \GKdim R/Q - \GKdim R/Q'$.

\end{prop}

\begin{proof}
First assume that $G$ is connected; so $\GSpec R$ consists of the $G$-stable primes of $R$.
In view of Schelter's catenarity theorem for $\Spec R$ \cite[6.3.43]{lR88}, it suffices to show
that any two neighbors $Q \subset P$ in $\GSpec R$ are also neighbors when viewed in $\Spec R$.
Passing to $R/Q$ we may assume that the algebra $R$ is prime and $P$ is a minimal nonzero member of
$\GSpec R$, and we need to show that $P$ has height $1$ in $\Spec R$. 
But $P \cap \cen(R)$ is a nonzero $G$-stable ideal of $\cen(R)$ and hence the Lie-Kolchin Theorem
provides us with a $G$-eigenvector $0 \neq z \in P\cap \cen(R)$. The ideal $P$ is a minimal prime
over $(z)$. For, if $(z) \subseteq P' \subset P$ for some $P' \in \Spec R$ then
$(z) \subseteq P'\byG \subset P$ and $P'\byG \in \GSpec R$, contradicting the fact that
$P$ is a minimal nonzero member of $\GSpec R$. Thus $P$ is minimal 
over $(z)$ as claimed, and the principal ideal theorem \cite[4.1.11]{jMcCjR87}
gives that $P$ has height $1$ as desired.

In general, let $G^0$ denote the connected component of the identity of $G$ and
put $\G = G/G^0$. Given $G$-primes $Q \subset Q'$ and
a saturated chain $Q=Q_0 \subset Q_1 \subset \dots
\subset Q_r=Q'$ in $\GSpec R$, we will show that $r = \GKdim R/Q - \GKdim R/Q'$.
To this end, write $Q_i = \bigcap_{x\in \G} x.P_i$ 
for suitable $P_i \in G^0\text{-}\!\Spec R$ as in the proof of Lemma~\ref{L:ratact}.
Since these intersections are finite intersections of $G^0$-primes, we can arrange that
$P_0 \subset P_1 \subset \dots \subset P_r$. This is a saturated chain in $G^0\text{-}\!\Spec R$. 
For, $P_i \subset P \subset P_{i+1}$ implies $Q_i = \bigcap_{x\in \G} x.P_i \subset \bigcap_{x\in \G} x.P \subset 
\bigcap_{x\in \G} x.P_{i+1} = Q_{i+1}$ since $\G$ is finite, which contradicts the fact that $Q_i$ and
$Q_{i+1}$ are neighbors in $\GSpec R$. By the first paragraph of the proof, the chain
$P_0 \subset P_1 \subset \dots \subset P_r$ is also saturated in $\Spec R$, and hence it has length
equal to $r = \GKdim R/P_0 - \GKdim R/P_r$ by Schelter's theorem. Since $\GKdim R/Q_i = 
\GKdim R/P_i$ by \cite[Corollary 3.3]{gKtL00}, the proof is complete.
\end{proof}


\begin{ack}
A preliminary version of some of these results was presented during the workshop at the 
CIMPA school  ``Braids in Algebra, Geometry and Topology'' in
Hanoi (January 17-28, 2011). 
The author would like to thank Ken Goodearl, Friedrich Knop and Claudio Procesi for their
valuable comments and suggestions.
\end{ack}


\bibliographystyle{amsplain}
\bibliography{../bibliography}


\end{document}